\documentclass[12pt]{amsart}
\usepackage{amscd}
\usepackage{amssymb}

\usepackage[all]{xy}

\newcommand{\bZ}{{\mathbb Z}}
\newcommand{\bF}{{\mathbb F}}

\newcommand{\bA}{{\mathbb A}}
\newcommand{\bC}{{\mathbb C}}
\newcommand{\bQ}{{\mathbb Q}}
\newcommand{\bP}{{\mathbb P}}
\newcommand{\bG}{{\mathbb G}}

\newcommand{\ra}{{\rangle}}
\newcommand{\la}{{\langle}}

\newtheorem{thm}{Theorem}[section]
\newtheorem{lemma}[thm]{Lemma}

\newtheorem{cor}[thm]{Corollary}
\newtheorem{prop}[thm]{Proposition}

\numberwithin{equation}{section}

\begin{document}

\title[cohomology of $GL_n(\bF_q)$]{
The etale cohomology of the general linear group over a finite field
and the Deligne and Lusztig variety}
 
\author{M.Tezuka and N.Yagita}

\address{ Department of mathematics, Faculty of Science, Ryukyu University, Okinawa, Japan\\
Department of Mathematics, Faculty of Education, 
Ibaraki University,
Mito, Ibaraki, Japan}
 
\email{ tez@sci.u-ryukyu.ac.jp,  yagita@mx.ibaraki.ac.jp, }

\begin{abstract}
Let $p\not =\ell$ be primes. 
We study the etale cohomology $H^{*}_{et}(BGL_n(\bF_{p^s});\bZ/{\ell})$
over 
the algebraically closed field $\bar \bF_p$
by using the stratification methods from Molina-Vistoli.
To compute this cohomology, we use the Delinge-Lusztig variety.
\end{abstract}

\maketitle

\section{Introduction}
Let $p$ and $\ell$ be  primes with 
$p \not =\ell$.
Let $G_n=GL_n(\bF_q)$ 
the general linear group over a finite
field $\bF_q$ with $q=p^s$.  
Then Quillen computed
the cohomology of this group in the famous paper [Qu].
\begin{thm} (Quillen [Qu])
Let $r$ be the smallest number such that $q^r-1=0\ mod(\ell)$.  
Then we have an isomorphism
\[ H^*(BG_n;\bZ/{\ell})\cong \bZ/\ell[c_r,...,c_{r[n/r]}]
\otimes \Delta(e_r,...,e_{r[n/r]})\quad (1.1)\]
where $|c_{rj}|=2rj,$ $|e_{rj}|=2rj-1$.
\end{thm}
To prove this theorem,  Quillen used the topological arguments, for example,
the Eilenberg-Moore spectral sequences, and spaces of the kernel of
the map $\psi^q-1$ defined by the Adams operation.
In this paper, we give an elementary algebraic proof for this theorem,
in the sense without using the above topological arguments.

By induction on $n$ and  the equivariant cohomology theory (stratified methods)
from Molina and Vistoli [Mo-Vi], [Vi],
we can compute the etale cohomology over $k=\bar \bF_p$, i.e.,
$H_{et}^*(BG_n;\bZ/{\ell})\cong (1.1)$. 
Then the base change 
theorem implies
the Quillen theorem.

The Molina and Vistoli stratified methods also work for the motivic cohomology.
Let $H^{*,*'}(-;\bZ/{\ell})$ be the motivic cohomology over the field 
$\bar \bF_p$ and $0\not =\tau\in H^{0,1}(Spec(\bar\bF_p);\bZ/{\ell})$.
\begin{thm} We have an isomorphism $H^{*,*'}(G_n;\bZ/{\ell})\cong
\bZ/{\ell}[\tau]\otimes(1.1)$ with degree
$deg(c_{rj})=(2rj,rj)$ and $deg(e_{rj})=(2rj-1,rj)$.
\end{thm}

To compute the equivariant cohomology, we consider the $G_n$-variety
\[ Q'=Spec(k[x_1,...,x_n]/((-1)^{n-1}det(x_i^{q^{j-1}})^{q-1}=1)),\]
and prove  $Q'/G_n\cong \bA^{n-1}$.  This implies the equivariant cohomology
\[H_{G_n}^*(Q'\times _{\mu_{q^n-1}}\bG_m;\bZ/p)\cong \Delta(f),\quad |f|=1.\]
The computation of the above isomorphism is the crucial point to compute
$H_{G_n}^*(pt.;\bZ/\ell)\cong H^*(BG_n;\bZ/\ell)$.

Let $G$ be a connected reductive algebraic group defined over a finite field ${\bF}_q$, $q=p^r$, let $F\colon G\rightarrow G$ be the Frobenius and let $G^F$ be the (finite) group of fixed points of $F$ in $G$, e.g., $GL_n^F=G_n$ in our notation. 
In the paper [De-Lu], Deligne and Lusztig
studied the representation theory of $G^F$ over fields of characteristic 0. The main idea is to construct such representations in the $\ell$-adic cohomology spaces 
$ H_c{}^*(\tilde X(\dot w),\bQ_{\ell})$
of certain algebraic varieties $\tilde X(\dot w)$ over ${\bF}_q$, on which $G^F$ acts.
(see $\S 6$ for the definition of $\tilde X(\dot w)$.)
 
For the $G=GL_n$ and$w=(1,\cdots,n)$, we see that 
$Q'\cong \tilde X(\dot w)$.
One of our theorems is to show $\tilde X(\dot w)/G_n\cong \bA^{n-1}$ for the above case by 
completely different arguments. 
The authors thank to Masaharu Kaneda and Shuichi Tsukuda for their useful suggestions.

\section{Dickson Invariants}

At first, we recall the Dickson algebra.  Let us write $G_n=GL_n(\bF_q)$.
The Dickson algebra is the invariant ring of a polynomial of $n$ variables
under the usual $G_n$-action, namely,
\[ \bF_q[x_1,..,x_n]^{G_n}=\bF_q[c_{n,0},c_{n,1},...,c_{n,n-1}]\]
where  each $c_{n,i}$ is defined by
\[\sum c_{n,i}X^{q^i}=\prod_{x\in \bF_q\{x_1,...,x_n\}}(X+x)=
\prod_{(\lambda_1,...,\lambda_n)\in\bF_q^{\times n}}(X+\lambda_1 x_1+...\lambda_n x_n)\]
Hence the degree $|c_{n,i}|=q^n-q^i$ letting $|x_i|=1$. 
Let us write  $e_n=c_{n,0}^{1/(q-1)}$, namely, 
\[ e_n=(\prod_{0\not =x\in \bF_q\{x_1,...,x_n\}}(x))^{1/(q-1)}
=\left|\begin{array}{cccc}
x_1&x_1^q&\ldots&x_1^{q^{n-1}}\\
x_2&x_2^q&\ldots&x_2^{q^{n-1}}\\
\multicolumn{4}{c}{\dotfill}\\
x_n&x_n^q&\ldots&x_n^{q^{n-1}}
\end{array}\right|.\]
Then  each $c_{n,i}$ is
written as
   \[ c_{n,s}=\left| \begin{array}{ccccc}
     x_1 &\ldots& \hat x_1^{q^s}& \ldots &  x_1^{q^{n}}\\
     x_2 & \ldots&\hat x_2^{q ^s}& \ldots &  x_2^{q^{n}}\\
           \multicolumn{5}{c}{\dotfill} \\
              x_n & \ldots &\hat x_n^{q^s} & \ldots &  x_n^{q^{n}}
       \end{array} \right|/e(x).\]
Note that the Dickson algebra for $SG_n=SL_n(\bF_q)$ is given as
\[ \bF_q[x_1,..,x_n]^{SG_n}=\bF_q[e_n,c_{n,1},...,c_{n,n-1}].\]

Let us write  $k=\bar \bF_p$. We consider the algebraic variety
\[F=Spec(k[x_1,...,x_n]/(e_n)).\]
We want to study  the $G_n$-space structure of $X=X(n)=\bA^n-\{0\}$ and 
$X(1)=X-F$.
For this, we consider the following variety
(the Deligne-Lusztig variety for $w=(1,...,n)$, see $\S 6$ for details)
\[Q=Spec(k[x_1,...,x_n]/(e_n-1)).\]

\noindent {\bf Example.}  When $q=p$ and $n=2$, we see
   \[ Q=\{(x,y)|x^py-xy^p=1\}\subset \bA^2,\]
   \[ F=\{(x,y)|x^py-xy^p=0\}=\cup _{i\in \bF_p\cup\{\infty\}}F_{i}\]
where $F_i=\{(x,ix)|x\in k\}$ and $F_{\infty}=\{(0,x)|x\in k\}$.
   
   The corresponding projective variety $\bar Q$ is written
\[\bar Q=Proj(k[x_0,...,x_n]/(e_n=x_0^{1+q+..+q^{n-1}})).\]
\begin{lemma}
Let us write $q(n)=1+q+...+q^{n-1}=(q^n-1)/(q-1)$.  Then we have an isomorphism
$Q\times_{\mu_{q(n)}}\bG_m\cong X(1)$ of varieties.
\end{lemma}
\begin{proof}
We consider the map 
\[ p:Q\times \bG_m\to X(1)\quad by\ (x,t)\mapsto tx.\]
We see
\[e_n(p(x,t))=e_n(tx_1,...,tx_n)=t^{1+q+...+q^{n-1}}e_n(x_1,...,x_n).\]
It is easily seen that this map is onto.
Moreover if $x\in Q$ and $t\in \mu_{q(n)}$, then $p(x,t)=tx\in Q$.
In fact $\mu_{q(n)}$ acts on $Q$.  Since $p(x,t)=p(tx,1)$, we have
the isomorphism in this lemma.
\end{proof} 
Remark 2.1 .It is immediate that the left $SG_{n}$-action and the right $\mu_{q(n)}$-action on Q is compatitive.i.e$(gx)\mu =g(x\mu) for g\in SG_
{n} and \mu \in \mu_{q(n)}.$
\begin{lemma}
We have $Q(\bF_q)=\emptyset$.
\end{lemma}
\begin{proof}
Let $(x_1,...,x_n)$ be a $\bF_q$-rational ponits.  Then $x_i^q=x_i$.
Hence we see
\[ e_n=\left|\begin{array}{cccc}
x_1&x_2&\ldots&x_n\\
x_1^q&x_2^q&\ldots&x_n^q\\
\multicolumn{4}{c}{\dotfill}\\
x_1^{q^{n-1}}&x_2^{q^{n-1}}&\ldots&x_n^{q^{n-1}}
\end{array}\right| =\left|\begin{array}{cccc}
x_1&x_2&\ldots&x_n\\
x_1&x_2&\ldots&x_n\\
\multicolumn{4}{c}{\dotfill}\\
x_1&x_2&\ldots&x_n
\end{array}\right|=0.\]
\end{proof}
\begin{lemma}
The group $SG_n$ acts on $Q$ freely.
\end{lemma}
\begin{proof}
Assume that there is  $0\not =g\in G_n$ such that 
\[gx=x\qquad for \ x\in Q\subset \bA^n. \]
 Then we can identify that $x$ is an eigen vector
for the (linear) action $g$ with the eigen value $1$.
Hence we can take $x=(1,0...,0)$ after some change of basis.
Of course $e_n(1,0...,0)=0$ so $x\not \in Q$.  This  is a contradiction.
\end{proof}

The group $SG_n$ acts freely on the (smooth) variety $Q$.
Hence $Q/SG_n$ exists as a variety and we have
\[ Q/SG_n=Spec(A^{SG_n})\quad for\ A=k[x_1,...,x_n]/(e_n-1).\]

\begin{thm}  We have an isomorphism
\[A^{SG_n}\cong k[c_{n,1},...,c_{n,n-1}]\quad i.e.,\ \ Q/SG_n\cong \bA^{n-1}.\]
\end{thm}
\begin{proof}
It is almost immediate
\[ k[c_{n,1},...,c_{n,n-1}]\subset A^{SG_n}.\]

The coordinate ring $\bar A$ of the Zariski closure $\bar Q$ of $Q$ in $\bP^{n}$
 is given as
\[\bar A=k[x_0,...,x_n]/(e_n=x_0^{1+q+..+q^{n-1}}).\]
Of course, the coordinate ring $\bar B$ of the closure of 
$Spec(k[c_{n,1},...,c_{n,n-1}])$  in $\bar Q$ is given as
\[\bar B=k[x_0,c_{n,1},...,c_{n,n-1}].\]

Next we compute the Poincare polynomials of $\bar A$ and $\bar A^{SG_n}$ ;
\[PS(\bar A)=(1-t^{1+q+...+q^{n-1}})/(1-t)^{n+1}=
(1+t+...+t^{q+...+q^{n-1}})/(1-t)^n,\]
\[PS(\bar B)=1/(1-t)(1-t^{|c_{n,1}|})...(1-t^{|c_{n,n-1}|})\]
\[=(1+t+...+t^{|c_{n,1}|-1})^{-1}...(1+t+...+t^{|c_{n,n-1}|-1})^{-1}/(1-t)^n.\]
Hence we get
\[PS(\bar A)/PS(\bar B)=(1+t+...t^{|c_{n,1}|-1})...(1+t+...t^{|c_{n,n-1}|-1})\]
\[  \qquad \qquad \qquad \times
(1+t+...+t^{q+..+q^{n-1}}).\]
Thus we know
\[rank(PS(\bar A)/PS(\bar B))=|c_{n,1}|\times ...\times |c_{n,n-1}|
\times (1+q+...+q^{n-1})\]
\[=(q^n-q^1)...(q^n-q^{n-1})((q^n-1)/(q-1))=|SG_n|.\]

On the other hand $c_{n,1},...,c_{n,n-1}$ is regular sequence in $\bar A$.
Hence $\bar A$ is $\bar B$-free,  that is
\[ \bar A=\bar B\{x_1,...,x_m\}\]
where $m=|SG_n|$ from the results using the Poincare polynomials above.

Let $\pi:Q\to Q/SG_n$ be the projection.
Since $\pi$ is etale, for all $x\in Q$, the local ring $O_x$ is
$O_{\pi(x)}$-free, and $rank_{O_{x'}}(O_x)=|SG_n|$.
Thus we get the desired result $A^{SG_n}=k[c_{n,1},...,c_{n,n-1}]$.
\end{proof}
Similarly, we can prove
\begin{thm}  Let $A'=k[x_1,...,x_n]/(e_n^{q-1}-1)$
and $Q'=Spec(A')$.  Then we have an isomorphism
\[(A')^{G_n}\cong k[c_{n,1},...,c_{n,n-1}]\quad i.e.,\ \ Q'/G_n\cong \bA^{n-1}.\]
\end{thm}
In $\S 7$ below, we give a complete different proof of the above theorem.

\section{equivariant cohomology}

For a smooth algebraic variety $X$ over $k=\bar \bF_p$, 
we consider the $mod$ $\ell$ etale cohomology
for $\ell\not =p$.
 Let us write simply
\[ H^*(X)=H_{et}^*(X;\bZ/\ell).\] 
Let $\rho: G\to W=\bA^n$ a faithful representation.
Let $V_n=W-S$ be an open set of $W$ such that $G$ act freely $V_n$
where $codim_WS> n\ge 2$.   Then the 
classifying space $BG$ of $G$ is defined as $colim_{n\to \infty} (V_n/G)$.
Let $X$ be a smooth $G$-variety.  Then we can define the equivariant
cohomology ([Vi], [Mo-Vi])
\[ H_G^*(X)=lim_nH_{et}^*(V_n\times _GX;\bZ/\ell).\]
Of course $H_G^*(pt.)=H^*(BG)=H^*_{et}(BG;\bZ/\ell)$.

One of the most useful facts in equivariant cohomology theories is the following localized exact sequence. Let $i:Y\subset X$ be a regular closed inclusion of $G$-varieties,
of $codim_X(Y)=c$ and $j:U=X-Y\subset X$.  Then there is a long exact sequence
\[ \to H_G^{*-2c}(Y)\stackrel{i_*}{\to}
         H_G^*(X)\stackrel{j^*}{\to} H_G^*(U)\stackrel{\delta}{\to}H_G^{*-2c+1}(Y)\to...\]
         
Now we apply the above exact sequence for concrete cases.
We consider the case $G=G_n=GL_n(\bF_q)$.  Recall
\[F=Spec(k[x_1,..., x_n]/(e_n^{q-1}))=\cup _{\lambda=(\lambda_1,...,\lambda_n)\not =0}(F_{\lambda})\]
where $F_{\lambda}=\{(x_1,...,x_n)|\lambda_1x_1+...+\lambda_nx_n=0\}\subset \bA^n.$  

Let $F(1)=F$ and $F(2)$ be the ($codim=1$) set of singular points in $F(1)$,
namely, $F(2)=\cup F_{\lambda,\mu}$ with
\[F_{\lambda,\mu}=\begin{cases} F_{\lambda}\cap F_{\mu}\quad if \ F_{\lambda}\not =F_{\mu}
\\ \emptyset\quad if \  F_{\lambda}= F_{\mu}.
\end{cases}\]   
Similarly, we define $F(i)$ as the variety defined by the set of $codim_{\bA^n}F(i)=i$. 
Let us write $X(i)=X-F(i).$
Thus we have a sequence of the algebraic sets
\[ F(1)\supset F(2)\supset ...\supset F(n)=\{0\}\supset F(n+1)=\emptyset,\]
\[ X-F(1)=X(1)\subset X(2)\subset...\subset X(n)=\bA^n-\{0\}\subset X(n+1)=\bA^n.\]
Therefore we have the long exact sequences
\[\to H_{G_n}^{*-2}(F(1)-F(2))\stackrel{i_*}{\to} H_{G_n}^*(X(2))\stackrel{j^*}{\to} 
H^*_{G_n}(X(1))
\stackrel{\delta}{\to}...,\]
\[ ...........................................\]
\[\to H_{G_n}^{*-2i}(F(i)-F(i+1))\stackrel{i_*}{\to} 
H_{G_n}^*(X(i+1))\stackrel{j^*}{\to} H^*_{G_n}(X(i))
\stackrel{\delta}{\to}...,\]
\[ ...........................................\]
\[\to H_{G_n}^{*-2n}(F(n)-F(n+1))\stackrel{i_*}{\to} 
H_{G_n}^*(X(n+1))\stackrel{j^*}{\to} 
H^*_{G_n}(X(n))
\stackrel{\delta}{\to}...\]

\begin{lemma} 
We have $H^*_{G_n}(X(1))\cong\Lambda(f)$ with $|f|=1$.
\end{lemma}
\begin{proof}
>From the $G_n$ version (but not $SG_n$) of Lemma 2.1,  we have
\[X(1)\cong Q'\times_{\mu_{q^n-1}}\bG_m.\]
Hence we can compute the equivariant cohomology from Theorem 2.5,Lemma 2.3 and Remark 2.1
\[H^*_{G_n}(X(1))\cong H^*(X(1)/G_{n})\]
\[\cong H^*(Q'/G_n\times_{\mu_{q^n-1}}\bG_m)\cong
  H^*(\bA^{n-1}\times_{\mu_{q^n-1}}\bG_m)\]
  \[\cong H_{\mu_{q^n-1}}(\bG_m)\cong \Lambda(f)\quad |f|=1.\]
  \end{proof}
  \begin{lemma}
  We have an isomorphism
  \[H_{G_n}^*(F(i)-F(i+1))\cong H^*(BG_{i})\otimes \Lambda(f)\]
  \end{lemma}
  \begin{proof}
  Each irreducible component of $F(i)$ is a $codim=i$ subspace, which is also identified
an element of the Grassmannian.  Hence we can write
\[  F(i)-F(i+1)\cong \amalg_{\bar g\in G_n/(P_{i,n-i})} g(\bA^{n-i}-F(1)')\]
where $g\in G_{n}$ is a representative element of $\bar g,$ $F(1)'=Spec(k[x_1,...,x_{n-i}]/(e_{n-i}^{q-1})$ and
$P_{i,n-i}$ is the parabolic subgroup
\[ P_{i,n-i}=(G_i\times G_{n-i})\ltimes U_{i,n-i}(\bF_q)
\cong \{\left(\begin{array}{cc} G_i& *\\
                                       0 & G_{n-i}\end{array}\right)
|*\in U_{i,n-i}(\bF_q)\}.\]
Since the stabilizer group of $X(1)'= \bA^{n-i}-F(1)'$is $P_{i,n-i}$, we note from [Vi] that $H_{G_{n}}^{*}(F(i)-F(i+1))\cong  H_{P_{i,n-1}}^*(X(1)')\cong H_{G_{i}\times G_{n-i}}^*(X(1)').$

Hence we can compute ( for $* < N$)
\[H^*_{G_n}(F(i)-F(i+1))\cong H^*(V_{N}'\times V_{N}''\times _{G_{i}\times G_{n-i}}X(1)'). \]
\[ \cong H^*((V_N'/G_i)\times V_N''\times_{G_{n-i}}X(1)'). \]
\[ \cong H_{G_i}^*\otimes H_{G_{n-i}}^*(X(1)').\]
Here $X(1)'$ is the ($n-i$)-dimensional version of $X(1)$,  and we identify
$V_N\cong V_N'\times V_N''$ where $G_i$ acts freely on $V_N'$ and so on.
>From the preceding lemma, we know $H_{G_{n-i}}^*(X(1)')\cong \Lambda(f)$.
\end{proof}

Let $r$ be the smallest number such that $q^r-1=0\ mod(\ell)$.
Recall that  
\[|G_n|=(q^n-1)(q^n-q)...(q^n-q^{n-1}).\]
Hence if $n<r$, then $H^*(BG_n)\cong\bZ/\ell$, and hence
$H_{G_n}^*(F(i)-F(i+1))\cong \Lambda(f)$ for $i\le n$.

The cohomology of $BGL_n$ is the same as that of $BGL_n(\bC)$, i.e.,
\[ H^*(BGL_n)\cong \bZ/\ell[c_1,...,c_n].\]
The Frobenius map $F$ acts on this cohomology by $c_i\mapsto q^{i}c_i$.
Recall that the Lang map induces a principal $G_{n}$-bundle
\[G_n\to GL_n\stackrel{L}{\to}GL_n\]
where $L(g)=g^{-1}F(g)$.  Thus we have a map
\[H^*(BGL_n)/((q^i-1)c_i)\cong \bZ/\ell[c_r,...,c_{[n/r]r}]
\to H^*(BG_n).\]

\begin{lemma} If $r=1$,  then
we have an isomorphism
 \[H^*(BG_n)
\cong 
 \bZ/\ell[c_1,...,c_{n}]
\otimes
\Delta(e_1,...,e_{n}).\]
 \end{lemma}
 \begin{proof}
 We prove by induction on $n$.  Assume that 
 \[H^*(BG_i)\cong 
 \bZ/\ell[c_1,...,c_{i}]\otimes\Delta(e_1,...,e_{i})\quad for \ i<n.\]
 
 We consider the long exact sequence
 \[\to H_{G_n}^{*-2i}(F(i)-F(i+1))\stackrel{i_*}{\to} H_{G_n}^*(X(i+1))\stackrel{j^*}{\to} H^*_{G_n}(X(i))
\stackrel{\delta}{\to}...\]
Here we use induction on $i$, and assume that
\[ H_{G_n}^*(X(i))\cong 
H_{G_{i-1}}^*\otimes \Lambda(e_i)\]
\[\cong \bZ/{\ell}[c_1,...,c_{i-1}]\otimes 
\Delta(e_1,...,e_{i-1})\otimes \Lambda(e_i).\]
(Letting $e_1=f$, we have the case $i=1$ from Lemma 3.1.)
>From the preceding lemma,  we still see
\[H_{G_n}^*(F(i)-F(i+1))\cong 
H_{G_{i}}^*\otimes \Lambda(f)\]
  \[\cong \bZ/{\ell}[c_1,...,c_i]\otimes \Delta(e_1,...,e_i)\otimes \Lambda(f).\]

In the above long exact sequence,the map $j^*$ is an epimorphism for $*<2i-1$,
because $H^{minus}(F(i)-F(i+1))=0$.  But $H_{G_n}^*(X(i))$ is multiplicatively generated
by the elements of $dim\le 2i-2$ and $e_{i}$.  By dimensional reason, we see
\[ \delta(e_{i})=1\quad or\quad\delta(e_{i})=0.\]
Of course if $\delta(e_{i})=0$, then $\delta=0$ for all $*\ge 0$. 

Consider the restriction map
$H_{G_n}^*(X(i+1))\to H_{G_i}^*(\bA^i)$
which is induced from $X(i+1)=\bA^n-F(i+1)\supset \bA^i$.
Since $|G_i|=(q^{ir}-1)q|G_{i-1}|$, the $\ell$-Sylow subgroup of $G_i$ is different from that
of $G_{i-1}$, (More precisely, $rank_{\ell}G_i>rank_{\ell}G_{i-1}$.)
So from the Quillen theorem, the Krull dimension of $H_{G_n}(X(i+1))$
is larger than that of $H_{G_n}^*(X(i))$. This fact implies
$i_*(1)=c_i$.  ( Let $p:V\to X$ be a $j$-dimensional bundle and $i:X\to V$ a section.  Then  the Chern class $c_j$ is defined as $i^*i_*(1)$.)
Thus we see $\delta(e_{i})=0$.  

Therefore we have the short exact  sequence
\[0\to 
H^*_{G_i}
\otimes\Lambda(f)
\stackrel{i_*}{\to}
 H^*_{G_n}(X(i+1))
\stackrel{j^*}{\to}
H^*_{G_{i-1}}\otimes\Lambda(e_{i})\to 0,\]
namely, we have an isomorphism
\[grH_{G_n}^*(X(i+1))\cong  
 \bZ/{\ell}[c_1,...,c_{i-1}]\otimes \Delta(e_1,...,e_{i})\]
\[  \otimes(\bZ/\ell[c_i]\{i_*(1)=c_i,i_*(f)\}\oplus \bZ/\ell\{1\}).\]
Let us write  $i_*(f)=e_{i+1}$.  Then $H^*_{G_n}(X(i+1))$ is the desired form
\[ H_{G_n}^*(X(i+1))\cong
\bZ/{\ell}[c_1,...,c_{i-1}]\otimes \Delta(e_1,...,e_{i})\]
\[  \otimes(\bZ/\ell[c_i]\{c_i,e_{i+1}\}\oplus \bZ/\ell\{1\})\]
\[ \cong
\bZ/{\ell}[c_1,...,c_{i}]
\otimes 
\Delta(e_1,...,e_{i})
\otimes \Lambda(e_{i+1}).\]
Thus we can see the desired result $H_{G_n}^*(X(n+1))\cong H^*(BG_n)$.
\end{proof}
{\bf  Remark.} In the above proof, to see $i_*(1)=c_i$ we used the Krull dimesion
(by Quillen).  However there is more natural argument (see Proposion 4.2 in the next section)
where the properties of the maximal torus $T(\dot w)$ are used.
\begin{thm} 
We have the isomorphism
 \[H^*(BG_n)
\cong 
 \bZ/\ell[c_{r},...,c_{[n/r]r}]
\otimes
\Delta(e_r,...,e_{[n/r]r}).\]
 \end{thm}
 \begin{proof}
 We prove the theorem by induction on $n$.  Assume that
 \[H^*(BG_i)\cong 
 \bZ/\ell[c_r,...,c_{[n/i]r}]\otimes\Delta(e_r,...,e_{[n/r]r})\quad for \ i<n.\]
 We also consider the long exact sequence
 \[\to H_{G_n}^{*-2i}(F(i)-F(i+1))\stackrel{i_*}{\to} H_{G_n}^*(X(i+1))\stackrel{j^*}{\to} H^*_{G_n}(X(i))
\stackrel{\delta}{\to}...\]
Here we use induction on $i$, and assume
$ H_{G_n}^*(X(i))\cong 
H_{G_{i-1}}^*\otimes \Lambda(e_i).$

>From Lemma 3.2,  we still see
\[H_{G_n}^*(F(i)-F(i+1))\cong 
H_{G_{i}}^*\otimes \Lambda(f).\]
  By dimensional reason, we see
$ \delta(e_{i})=1\quad or\quad\delta(e_{i})=0.$

Now we consider the case $r\ge 2$ and
$mr<i\le mr+r-1.$
This case we still assume
\[ H_{G_i}^*\cong H_{G_{i-1}}^*\cong H_{G_{mr}}^*
\cong \bZ/\ell[c_r,...,c_{mr}]\otimes \Delta(e_r,...,e_{mr}).\]
Hence the  above exact sequence is written as
\[\to 
H^*_{G_{mr}}
\otimes\Lambda(f)
\stackrel{i_*}{\to}
 H^*_{G_n}(X(i+1))
\stackrel{j^*}{\to}
H^*_{G_{mr}}\otimes\Lambda(e_{i})\to ....\]

The $\ell$- Sylow subgroup of $G_i$ and $G_{i-1}$ are the same,
and hence $c_i=0$ in $H^*_{G_i}$. (See also Proposition 4.2 below.)
This means  
$  \delta(e_i)=1$ (Of course $\delta(1)=0$).

Hence we have the isomorphism
\[H_{G_n}^*(X(i+1))\cong  H^*_{G_{mr}}\{1,i_*(f)\}\cong H_{G_{mr}}^*\{1,e_{i+1}\}
\cong H_{G_i}^*\otimes \Lambda(e_{i+1}).\]
Other parts of the proof are almost the same as in the case $r=1$.
\end{proof}

\section{maximal torus and $SL_n$}

Let $r$ be the smallest positive integer such that
$q^r-1=0\ (mod(\ell))$.
Let $w=(1,2,..,r)\in S_r$
and $G_r=GL_r(\bF_q)=GL_r^F$ for the Frobenius map $F:x\mapsto x^q$.
For a matrix $A=(a_{i,j})\in GL_n$, the adjoint action is given as 
\[ad(w)F(A)=wFw^{-1}(a_{i,j})=(b_{i,j})\quad with \ b_{i,j}=a_{i-1,j-1}^q.\]
 Let  $T(w)$ be the maximal torus $T^*\subset GL_r$, for which the Frobenius is 
given as $ad(w)F$ (see the next section for details)
so that
\[T(w)^F=\{t\in T^*|ad(w)F(t)=t\}\]
 \[\cong \{x\in \bF_{q^r}^*|(x,x^q,...,x^{q^{r-1}})\in T^*\}\cong \bF_{q^r}^*.\] 
Take $H^*(BT^*)\cong \bZ/\ell[t_1,..,t_ r]$.  Let $i:T(w)^F\subset T^*$.
Then we can take the ring generator $t\in H^2(BT(w)^F)$ such that
$i^*t_i=q^{i-1}t$. 
\begin{lemma} The following map is injective
\[ H^*(BGL_r)/((q^i-1)c_i)\cong \bZ/{\ell}[c_r]\to H^*(BG_r).\]
\end{lemma}
\begin{proof}
It is enough to prove that for the  map 
\[i^*:H^*(BGL_r)\to H^*(BG_r)\to H^*(BT(w)^F)\cong H^*(\bF_{q^r}^*),\]
we can see $i^*c_1=...=i^*c_{r-1}=0$, and $ i^*c_r=(-1)^rt^r$.   

Let $s_i$ be  the $i$-th elementary symmetric function
of variables $t_1,...,t_r$, namely, 
\[(X-t_1)(X-t_2)...(X-t_r)=X^r+s_1X^{r-1}+...+s_r.\]
Since $i^*(t_i)=q^{i-1}t$,  we see that      
\[(X-t)(X-qt)...(X-q^{r-1}t)=X^n+i^*(s_1)X^{r-1}+...+i_*(s_r) .\]
On the other hand,  the polynomial $X^r-t^r$ has its roots  
$X=t,qt,...,q^{r-1}t$.  Hence we see that the above formula is 
$X^r-t^r.$  It implies the assertion above.
\end{proof}
\begin{prop} The following map is injective
\[ H^*(BGL_n)^F\cong \bZ/{\ell}[c_r,...,c_{[n/r]r}]\to H^*(BG_r).\]
\end{prop}
\begin{proof}
Let $k=[n/r]$.  let us take
\[w=(1,...,r)(r+1,...,2r)...((k-1)r+1,...,kr).\]
We consider the map
\[i^*:H^*(BGL_n)\to H^*(BG_n)\to H^*(BT(w)^F)
      \cong H^*(B(\bF_{q^r}^*\times ... \times\bF_{q^r}^*)).\] 
We chose $t_i\in H^2(BT)$ ($1\le i\le n$) and 
$t_j'\in H^2(BT(w)^F)$ $(1\le j\le k$) such as
$i^*t_1=t_1', i^*t_2=qt_1',...$.
Then the arguments similar to the proof of the preceding lemma, we have
\[X^n+i^*(c_1)X^{r-1}+...+i_*(c_r)=(X^r\pm (t_1')^r)...(X^r\pm (t_k')^r) .\] 
Thus we get the result.
\end{proof}

Now we consider the case $G=SL_n$.
Write $SL_n(\bF_{q})$ by $SG_n$.
\begin{lemma} If $r\ge 2$, then, the following map is injective
\[ H^*(BSL_r)^F\cong \bZ/{\ell}[c_r]\to H^*(BSG_r).\]
\end{lemma}
 \begin{proof}
Let $w=(1,...,r)$ and recall $q(r)=1+q+...+q^{r-1}$.  Then
the maximal torus of $SG_r$ is written
\[ ST^*(w)^F\cong \{t\in F_{q^r}^*|(x,...,x^{q^{r-1}})\in T^*,\ x^{q(r)}=1\}
\cong \bZ/q(r).\]
We consider the map as the case $G_r$
\[i^*:H^*(BSL_r)\to H^*(BSG_r)\to H^*(BST(w)^F)\]
 \[     \cong H^*(B\bZ/q(r))\cong \bZ/\ell[t]\otimes \Lambda(v).\]
Let us write $H^*(BST^*)\cong \bZ/\ell[t_1,...,t_r]/(t_1+..+t_r)$.
Then we also see that $i^*(t_i)=q^{i-1}t$ ( note $\sum q^{i-1}=q(r)=0\in \bZ/\ell$).
The arguments in the proof of Lemma implies this lemma.
\end{proof}
\begin{prop} For the case $r\ge 2$, the following map is injective
\[ H^*(BGL_n)^F\cong \bZ/{\ell}[c_r,...,c_{[n/r]r}]\to H^*(BSG_n).\]
When $r=1$, the map $\bZ/\ell[c_2,...,c_n]\to H^*(BSG_n)$ is injective.
\end{prop}
\begin{proof} 
The maximal torus of $SG_n$ is written
\[ ST^*(w)^F\cong \{t\in F_{q^r}^*|(x_1,...,x_1^{q^{r-1}},...,x_k,...,x_k^{q^{r-1}}))\in T^*,
\ (x_1...x_k)^{q(r)}=1\}.\]
We can get the result as  the case $G_n$.  When $r=1$, note that 
$c_1=t_1+...+t_n=0$ still in $H^*(BST^*)$.
\end{proof}
\begin{thm} For the case $r\ge 2$, we have the isomorphism
$H^*(BSG_n)\cong H^*(BG_n)$.
When $r=1$, we have
\[ H^*(BG_n;\bZ/{\ell})\cong \bZ/\ell[c_2,...,c_{n}]
\otimes \Delta(e_2,...,e_{n}).\]
\end{thm}
\begin{proof}[An outline of the proof.]
Almost arguments work as the case $G_n$.  For example,
in the proof of Lemma 3.2,  for $G=G_n$, we showed
\[  F(i)-F(i+1)\cong G_n/(P_{i,n-i})\times (\bA^{n-i}-F(1)')\]
where $P_{i,n-i}$ is the parabolic subgroup
$(G_i\times G_{n-i})\ltimes U_{i,n-i}$.
We must consider the $SG_n$-version
\[ SG_n/S(G_i\times G_{n-i})\ltimes U_{i,n-i}(\bA^{n-i}-F(1)').\]
Here we can reduce  $S(G_i\times G_{n-i})$
to the case $G_i\ltimes SG_{n-i}$.
Then the inductive arguments work also this case.
\end{proof}

\section{motivic cohomology}

In this section, we consider the motivic version of preceding section.
Let us write
\[H^{*,*'}_G(X)=H^{*,*'}_G(X;\bZ/p)\]
the (equivariant) motivic
cohomology over the field $k=\bar \bF_p$.
Then we have the long exact sequence
\[\to H_{G_n}^{*-2i,*'-i}(F(i)-F(i+1))\stackrel{i_*}{\to} 
H_{G_n}^{*,*'}(X(i+1))\stackrel{j^*}{\to} H^{*,*'}_{G_n}(X(i))
\stackrel{\delta}{\to}.\]
However we note the following fact: the projection
\begin{align*}
V_{N}''\times_{G_{n-i}}(\bA^{n-i}-F(1)')\to &\bA^{n-i}-F(1)'/G_{n-i} \\
&\cong\bA^{n-i-1}\times_{\mu_{{q}^{n-i}-1}} \bG_{m}\to \bG_{m}/\mu_{{q}^{n-i}-1}\cong \bG_{m}
\end{align*}
is an $\bA^1$-homotopy equivalence when we replace $V_{N}''$ as a suitable large $G_{n-i}$-vector space. Then Lemma 3.2 holds for the motivic cohomology.
Then the most arguments in the preceding sections also work
for the motivic cohomology
with the degree
\[ deg(c_i)=(2i,i),\qquad deg(e_i)=(2i-1,i).\]
Thus we get Theorem 1.2 in the introduction.

\section{the Deligne-Lusztig theory}

Let $G$ be a connected reductive algebraic group defined over a finite field ${\bF}_q$, $q=p^r$, let $F\colon G\rightarrow G$ be the Frobenius map and let $G^F$ be the (finite) group of fixed points of $F$ in $G$. 

In the paper [De-Lu], Deligne and Lusztig
studied the representation theory of $G^F$ over fields of characteristic 0. The main idea is to construct such representations in the $\ell$-adic cohomology spaces of certain algebraic varieties $\tilde X(\dot w)$ over ${\bF}_q$, on which $G^F$ acts. 

Fix a Borel subgroup $B^*\subset G$ and a maximal ${\bF}_q$-split torus $T^*\subset B^*$, both defined over ${\bF}_q$. Let $W$ be the Weyl group of $T^*$ and 
\[G=\bigcup_{w\in W}B^*\dot wB^*\qquad  (disjoint\ union)\]
 be the Bruhat decomposition, $\dot w$ being a representative of $w\in W$ in the normalizer of $T^*$. Let $X$ be the variety of all Borel subgroups of $G$. This is a smooth scheme over ${\bF}_q$, on which the Frobenius element $F$ acts. Any $B\in X$ is of the form $B=gB^*g^{-1}=adgB^*$, where $g\in G$ is determined by $B$ up to right multiplication by an element of $B^*$. Let $X(w)\subset X$ be the locally closed subscheme consisting of all Borel subgroups $B=gB^*g^{-1}$ such that $g^{-1}F(g)\in B^*\dot wB^*$, namely,
 \[ (6.1)\quad X(w)=\{g\in G|g^{-1}F(g)\in B^*\dot wB^*\}/B^*\]
\[    \cong \{g\in G|g^{-1}F(g)\in \dot wB^*\}/(B^*\cap ad\dot wB^*).\]
( Borel groups $ad(g)B^*$ and $ad(g)FB^*$ are called in relative position $w$
if $g\in X(\dot w)$.)

 For any $w\in W$, let $T(w)$ be the torus $T^*$, for which the Frobenius 
map is given by $ad(w)F$
so that
\[(6.2)\quad T(w)^F=\{t\in T^*|ad(w)F(t)=t\}.\]
  Hence $T(w)^F$ is isomorphic to the set of ${\bF}_q$-points of a torus $T(w)\subset G$, defined over ${\bF}_q$. 

Let $U^*$ be the unipotent radical of $B^*$. For any $B\in X$ let $E(B)=\{g\in G|gB^*g^{-1}=B\}/U^*$. The Frobenius map induces a map $F\colon E(B)\rightarrow E(F(B))$. Let $E(B,\dot w)=\{u\in E(B)|F(u)=u\dot w\}$. For $B\in X(w)$ the 
sets $E(B,\dot w)$ are the fibers of a map $\pi\colon\tilde X(\dot w)\rightarrow X(w)$, where $\tilde X(\dot w)$ is a right principal homogeneous space of $T(w)^F$ over $X(w)$. The groups $G^F$ and $T(w)^F$ act on $\tilde X(\dot w)$ and these actions commute. 
Thus we have the isomorphism
\[ (6.3)\quad \tilde  X(\dot w)\cong
 \{g\in G|g^{-1}F(g)\in \dot wU^*\}/(U^*\cap ad\dot wU^*).\]

Now let $\ell$ be a prime distinct from $p$, and $\bQ_{\ell}$ be the algebraic closure of the field of $\ell$-adic numbers. Deligne-Lusztig  consider the actions of $G^F$ and $T(w)^F$ on the $\ell$-adic cohomology $H_c{}^*(\tilde X(\dot w),\bQ_{\ell})$ with compact support. For any $\theta\in {Hom}(T(w)^F,\bQ_{\ell})$,  let $H_c^*(\tilde X(\dot w),\bQ_{\ell})_\theta$ be the subspace of $H_c{}^\ast(\tilde X(\dot w),\bQ_{\ell})$ on which $T(w)^F$ acts by $\theta$. This is a $G^F$-module. 

The main subject of the paper [De-Lu] 
is the study of virtual representations $R^\theta(w)=\sum_i(-1)^iH_c{}^i(\tilde X(\dot w),\bQ_{\ell})_\theta$ (it can be shown that the right hand side is independent of the lifting $\dot w$ of $w$). 

\noindent {\bf Example.} (See 2.1 in [De-Lu].)
Let $V$ be an $n$-dimensional vector space over $k$ and put
$G=GL(V)$. We may take a basis such that a maximal torus $T\cong \bG_m^n$ and
the Weyl group $W\cong S_n$; the symmetric group of $n$-letters.
Then $X=G/B$ is the space of complete flags 
\[D \ : \ D_0=0\subset D_1\subset ...\subset D_{n-1}\subset D_n=V\]
with $dimD_i=i$.  The space $E=G/T$ is the space of complete flags marked by nonzero
vector $e_i\in D_i/D_{i-1}$,  where $T$ acts on $E$ by $(D,(e_i))(t_i)=(D,(t_ie_i))$.

Let $w=(1,...,n)$. Then two flags $D'$ and $D''$ are relative position $w$ 
(for details see 1.2 in [De-Lu])  if and only
if 
\[ D_i''+D_i'=D_{i+1}'\ (1\le i<n-1),\quad D_{n-1}''+D_1^i=V.\]
Hence $D$ and $FD$ are  in relative position $w$, if and only if
\[D_1\subset D_1+FD_1\subset D_1+FD_1+F^2D_1\subset ...\]
and $V=\oplus^{n+1}F^iD_1$.  A marking $e$ of $F$ is given such that
$F(e)=e\cdot \dot{w}$ if and only if
\[ e_2=F(e_1)(mod(e_1)),\ \ ...,\ \ e_n=F^{n-1}(e_1)(mod(e_1,...,F^{n-2}(e_1))\]
\[and \qquad e_1=F^n(e_1)(mod(e_1,...,F^{n-1}(e_1));\]
Hence the mark $e$ is defined by $e_1\in D_1$ with the condition that
\[F( e_1\wedge F(e_1)\wedge...\wedge F^{n-1}(e_1))=
(-1)^{n-1}(e_1\wedge F(e_1)\wedge...\wedge F^{n-1}(e_1)).\]
If $(x_i)$ are the coordinate of $e_1$,  the above condition can be
rewritten
\[(6.4)\quad (-1)^{n-1}(det(x_i^{q^{j-1}})_{1\le i,j\le n})^{q-1}=1.\]
Hence the map $(D_1,e_1)$ induces an isomorphism of
$\tilde X(\dot{w})$ with the affine hypersurface (6.4).
Note that this hypersurface is stable under $x\mapsto tx$ for $t\in F_{q^n}^*$,  and this is the action of $T(w)^F$.

Recall that $(det(x_i^{q^{j-1}})_{1\le i,j\le n})$ is written by $e_n$
in $\S 2$. Thus we have 
\begin{thm}
The variety $Q'$ in Theorem 2.5 in $\S 2$ is isomorphic to
$\tilde X(\dot w)$.
\end{thm}
In the next section, we will give a complete different proof of
the above theorem.

\section{The Deligne-Lusztig variety $\tilde X(\dot{w_n})$}
In 1.11.4 in [De-Lu],  Deligne and Lustig prove the following
theorem
\begin{thm}
\[G_n\setminus\tilde X(\dot w_n)\cong U^*/(U^*\cap ad(\dot w_n)U^*).\]
\end{thm}
We will give a complete different proof of the above theorem 
and Theorem 2.4 by using
Dickson invaraints for $G=GL_n(\bF_q)$ and $w_n=(1,...,n)$.

Take an adequate basis of the $n$ dimensional vector space
such that 
\[ w_n=\left( \begin{array}{cccc}
     0 & 0 & \ldots  & 1\\
     1 & 0 & \ldots  & 0\\
              \multicolumn{4}{c}{\dotfill} \\
     0 & \ldots & 1 & 0\\
     \end{array} \right),
\qquad U^*=\{ \left( \begin{array}{cccc}
     1 & * & \ldots  & *\\
     0 & 1 & \ldots  & *\\
      \multicolumn{4}{c}{\dotfill} \\
     0 &  \ldots & 0 &1
\end{array}\right) | *\in \bar F_p\}. \]
Let $x_{i,j}(a)=1+ae_{i,j}$ where $e_{i,j}$ is the elementary matrix with $1$ in 
$(i,j)$-entry and $0$ otherwise.  Then $U^*$ is generated by $x_{i,j}(a)$,
\[ U^*=\la x_{i,j}(a)| 1\le i<j \le n\ |\ a\in \bar F_p\ra\]
with the relation
\[ x_{i,j}(a)x_{i,j}(b)=x_{i,j}(a+b),\quad [x_{i,j}(a),x_{k,l}(b)]=\delta_{j,k}x_{i,l}(ab).\]
Note $ad(w)x_{i,j}(a)=wx_{i,j}(a)w^{-1}=x_{i+1,j+1}(a)$ identifying $i,j\in \bZ/n$.  Hence
\[InU^*=U^*\cap ad(w)U^*=\la x_{i,j}|x_{1,j}=0\ra\]
and $ad(w^{-1})InU^*=\la x_{i,j}|x_{i,n}=0\ra$,  that is
\[ InU^*=\left( \begin{array}{ccccc}
     1 & 0 & 0&\ldots  & 0\\
     0 & 1 &  *&\ldots & *\\
             \multicolumn{4}{c}{\dotfill}&* \\
     0 & 0 & \ldots  & 0&1\\
     \end{array} \right),
\qquad ad(w^{-1})InU^*= \left( \begin{array}{ccccc}
     1 & * & \ldots&*  & 0\\
     0 & 1 & \ldots&*  & 0\\
    \multicolumn{4}{c}{\dotfill} &\cdot \\
     0 & 0 & \ldots &0 &1
\end{array}\right) . \]

In Theorem 7.1, the $InU^*$ action on $U^*$ is given by the following
$\rho$ (see 1.11.4 in [De-Lu])
\[ \rho(u)v=ad(\dot w^{-1}_n)(u)vF(u^{-1})\quad for \ u\in InU^*,\ v\in U^*.\]
\begin{lemma} There is an isomorphism
\[U^*/\rho(InU^*)\cong \la x_{ij}(a)|x_{i,j}=0\ if \ j\not =n\ra\]
\[ =\{ \left( \begin{array}{cccccc}
     1 & 0 &\ldots  & 0&d_1\\
                \multicolumn{4}{c}{\dotfill}&* \\
       0 & 0&\ldots & 1&d_{n-1}\\
     0 & 0 & \ldots  & 0&1\\
     \end{array} \right)\in U^*\ |\ d_1,...,d_{n-1}\in \bar F_p\}.\]
     \end{lemma}
     \begin{proof}
    We consider the $\rho$-action when the case $u=x_{i,j}(a)$ and $v=x_{k,l}(b)$, namely,
    \[ \rho(u)v=ad(\dot w^{-1})(x_{ij}(a))x_{k,l}(b)F(x_{i,j}(a)^{-1})\]
    \[ =x_{i-1,j-1}(a)x_{k,l}(b)x_{i,j}(-a^q).\]

For roots $x_{i,j}$ and $x_{i',j'}$, we define an order $x_{i,j}<x_{i',j'}$
if $i<i'$ or $i=i'$, $j<j'$. Then  any $v\in U^*$ is uniquely written
by the product $\Pi x_{i,j}(b_{i,j})$ when we fix the above order in the product.
For any $a\in U^*$, let $x_{i_0,j_0}$ be the  minimal root of $v$
such that $x_{i_0,j_0}(b_{i_0,j_0})\not =0, j_0<n$.
   
   Take $i=i_0+1$, $j=j_0+1$ and $a=-b_{i_0,j_0}$.  Then the equation 
 \[ \rho(u)v=ad(\dot w^{-1})(x_{ij}(a))(\Pi x_{k,l}(b))F(x_{i,j}(a)^{-1})\]
    \[ =x_{i_0,j_0}(-b_{i_0,j_0})(\Pi x_{k,l}(b_{i,j}))x_{i_0+1,j_0+1}(-a^q)\]
\[   =(\Pi_{(i_0,j_0)<(k,l)} x_{k,l}(b_{i,j}))x_{i_0+1,j_0+1}(-a^q)\]
     implies that a nonzero minimal root of $\rho(u)v$ is larger than $(i_0,j_0)$.
   Repeating this process, there exists $u\in InU^*$ such that 
    $\rho(u)v=\Pi_{i=1}^{n-1}x_{i,n}(d_i)$.
    
 But all nonzero elements in the right hand side group in this lemma
  are not in $Im(\rho(u))v$ for $u\not =1$.  Hence we have the lemma.
\end{proof} 

Recall that we can identify
\[ Q'=\{x=\left( \begin{array}{cccc}
     x_1 & x_1^q& \ldots &  x_1^{q^{n-1}}\\
     x_2 & x_2^q & \ldots &  x_2^{q^{n-1}}\\
           \multicolumn{4}{c}{\dotfill} \\
              x_n & x_n^q & \ldots &  x_n^{q^{n-1}}
       \end{array} \right) \in GL_n| |x|^{q-1}=det(x)^{q-1}=1\}. \]
  \begin{thm}   We can define the map $f:Q'\to U^*/(\rho(InU^*)$ 
by $x\mapsto \dot w^{-1}_nx^{-1}Fx$, in fact,
\[ f(x) =\left( \begin{array}{cccccc}
     1 & 0 &\ldots  & 0& c_{n,1}\\
                \multicolumn{4}{c}{\dotfill}&* \\
       0 & 0&\ldots & 1&{c_{n,n-1}}\\
     0 & 0 & \ldots  & 0&1\\
     \end{array} \right) \]
     where $c_{n,i}=c_{n,i}(x_1,...,x_n)$ is the Dickson element defined in $\S 2$.
  This map also induces the isomorphism
\[ G_n\setminus Q'\cong U^*/(\rho(InU^*)) \cong   Spec(k[c_{n,1},...,c_{n,n-1}]) \qquad (so\ Q'\cong \tilde X(\dot w_n)).\] 
\end{thm} 
\begin{proof}      
 Let us write
\[e_n\left( \begin{array}{cccc}
     i_1 & i_2& \ldots &  i_n\\
    j_1 & j_2 & \ldots &  j_n\\
            \end{array} \right)
=\left| \begin{array}{cccc}
     x_{j_1}^{q^{i_1}} & x_{j_1}^{q^{i_2}}& \ldots &  x_{j_1}^{q^{i_n}}\\
     x_{j_2}^{q^{i_1} }& x_{j_2}^{q^{i_2}}& \ldots &  x_{j_2}^{q^{i_n}}\\
     
           \multicolumn{4}{c}{\dotfill} \\
          x_{j_n}^{q^{i_1}} & x_{j_n}^{q^{i_2}}& \ldots &  x_{j_n}^{q^{i_n}}\\  
                 \end{array} \right| \]     
 so that   $e_n\left( \begin{array}{cccc}
     0 & 1& \ldots &  n-1\\
    1 & 2 & \ldots &  n\\
            \end{array} \right)=e(x)=|x|.$    
Then the $(j,i)$ cofactor of the matrix $x$ is expressed as 
\[D_{j,i}=(-1)^{i+j}e_{n-1}\left( \begin{array}{cccccc}
     0 & 1& \ldots& \hat{i-1}& \ldots & n-1\\
     1 & 2& \ldots& \hat j& \ldots & n
            \end{array} \right).\]
            By Clamer's theorem, we know
            \[ x^{-1}=|x|^{-1}(D_{j,i})^t=|x|^{-1}(D_{i,j}).\]
            
            Let us write
            $(B_{i,j})=|x|x^{-1}F(x)$.
            Then we can compute
            \[ B_{s,t}=\sum D_{s,k}x(k,t)^q=\sum D_{s,k}x_k^{q^t}
              \qquad (where\ x(k,l)=(k,l)-entry \ of\ x)\]
              \[ =\left| \begin{array}{ccccc}
     x_1 &\ldots& \stackrel{s}{x_1^{q^t}}& \ldots &  x_1^{q^{n-1}}\\
     x_2 & \ldots&x_2^{q ^t}& \ldots &  x_2^{q^{n-1}}\\
           \multicolumn{5}{c}{\dotfill} \\
              x_n & \ldots &x_n^{q^t} & \ldots &  x_n^{q^{n-1}}
       \end{array} \right|.\]
       This element is nonzero only if $t=s-1$ or $t=n$.
       If $t=s-1$, then the above element is $|x|$.
       If $t=n$,  then the above element is indeed, $(-1)^{n-s}|x|c_{n,s-1}$ by the definition of
       the Dickson elements as stated in $\S 2$.
       Thus we have
       \[x^{-1}F(x)=|x|^{-1}(B_{st})=
       \left( \begin{array}{cccccc}
     0 & 0 &\ldots  & 0& c_{n,0}\\
     1 & 0&\ldots & 0&{c_{n,1}}\\
                \multicolumn{4}{c}{\dotfill}&* \\
      
     0 & 0 & \ldots  &1&c_{n,n-1}\\
     \end{array} \right). \]
     Here $c_{n,0}=1$ and   
       acting $\dot w^{-1}_n$, we have the desired result for $f(x)$.\\
       
       We will show that $f$ is an isomorphism.\\
       We note that $f$ is decomposed into
\[
	\xymatrix{
	G_n\backslash GL_n \ar[r]^-{\bar{L}} 
	& GL_n \ar[r]^-{\dot w_n^{-1}} 
	& G_n	\backslash GL_n \\
	G_n\backslash Q' \ar[rr]_{f} \ar[u]^-{incl.} &  
	& U^*/\rho(InU^*) \ar[u]_-{incl.}
	}
\]
       where $L(x)=x^{-1}F(x)$.
       
       Since the Lang map is separable, so is $f$.
       We see that $f$ is injective from the diagram. To show that $f$ is an isomorphism, it is enough to see that
       $f:Q'\to U^*/\rho(InU^*)$
       is surjective.\\
       When we consider $Q'$ as a subvariety of $\bA^{n}$, the above $f$ is identified with a map 
       $g|Q'$, where 
       $g:\bA^{n}\to\bA^{n-1}$
       is defined by
       $g(x)=(c_{n,1}(x),...,c_{n,n-1}(x)).$\\
       Then the surjectivity follows from the following lemma:
  
\end{proof} 

\begin{lemma}
Let $(f_{1},...,f_{n})$ be a homogeneous regular sequence of $k[x_{1},...,x_{n}]$. Then the associated map $f:\bA^{n}\to\bA^{n}$ is surjective. It means that\\
$f':V(f_{1}-a)\to\bA^{n-1}$ is surjective for $a$$\in$ k where $f'=pr(f|V(f_{1}-a))$where $pr:\bA^{n}\to\bA^{n-1}$ is the projection $pr(x_{1},...,x_{n})=(x_{1},...,x_{n-1})$
\end{lemma}
\begin{proof}
We consider the inclusion $i:\bA^{n}\subset\bP^{n}$ defined by 
$i(x_{1},...,x_{n})=[x_{1},...,x_{n},1]$, and denote the coordinate of $\bP^{n}$by$[u_{1},...u_{n},z]=[u,z]$.\\
We denote by$\tilde f:\bP^n\to \bP^n$the rational map extended from $f$ and denote by $d_{i}$ the degree of $f_{i}$.\\

For $\alpha \in \bA^{n} $,we see that $\tilde f^{-1}(\alpha)$is given by 
\[V_{+}(f_{1}(u)-\alpha_{1}z^{d_{1}},...,f_{n}(u)-\alpha_{n}z^{d_{n}}),when \alpha = (\alpha_{1},...,\alpha_{n}).\]\\
Then $\tilde f^{-1}(\alpha)\neq \phi$ by the Bezout theorem.
Since $(f_{1},...,f_{n})$ is a homogeneous regular sequence, we see that $V(f_1,...,f_n)=\{0\}$. It implies that
\[\tilde f^{-1}(\alpha)\cap V_{+}(z)=\{[u_{1},...,u_{n},1]|f_{1}(u)=\dots=f_{n}(u)=0\}=\phi.\]
We have $f^{-1}(\alpha)=\tilde f^{-1}(\alpha)\neq\phi$.Hence $f$ is surjective.\\

\end{proof}

Hence we know
\[\tilde X(\dot w)\cong \{(x_1,...,x_n)\in \bA^n|e(x_1,...,x_n)^{q-1}=|x|^{q-1}=1\}\]
\begin{thm} There is an isomorphism of varieties
\[X(1)\cong \tilde X(\dot w_n)\times_{T(\dot w_n)^F}\bG_m.\]
\end{thm}
%
%
\begin{cor} We have isomorphisms
\[G_n\backslash
X(1)\cong G_n\backslash
(\tilde X(\dot w_n)
\times_{T(\dot w)^F}\bG_m\cong \bA^{n-1}\times \bG_m.\]
\end{cor}

\end{document}